\documentclass[12pt]{amsart}
\usepackage{amsthm,amsmath,amssymb,mathrsfs,amsbsy,bbm}
\usepackage{comment}
\usepackage{fullpage}
\usepackage[utf8]{inputenc}
\usepackage[english]{babel}
\usepackage[colorlinks,linkcolor = red]{hyperref}
\usepackage[bottom]{footmisc}
\usepackage[
backend=biber,
style=alphabetic,
sorting=ynt
]{biblatex}
 
\nocite{*}

\hypersetup{ linktoc=page }
\newtheorem{theorem}{Theorem}[subsection]
\newtheorem*{main theorem}{Main Theorem}
\newtheorem{definition}[theorem]{Definition}
\newtheorem{claim}[theorem]{Claim}
\newtheorem{lemma}[theorem]{Lemma}

\newtheorem{corollary}[theorem]{Corollary}
\newtheorem{notion}[theorem]{Notion}

\newtheorem{remark}[theorem]{Remark}
\newtheorem{conjecture}[theorem]{Conjecture}

\newcommand{\Gl}[3]{\mathrm{GL}_{#1}(\mathbb{#2}_{#3})}
\newcommand{\bb}[1]{\mathbb{#1}}
\newcommand{\rmf}[1]{\mathrm{#1}}
\newcommand{\jr}[2]{\mathcal{J}_{#1}(#2)}
\newcommand{\InPr}[2]{\left < {#1}, {#2} \right >}

\catcode`\@=11
\newdimen\cdsep
\def\cdstrut{\vrule height .6\cdsep width 0pt depth .4\cdsep}
\def\@cdstrut{{\advance\cdsep by 2em\cdstrut}}

\def\arrow#1#2{
	\ifx d#1
	\llap{$\scriptstyle#2$}\left\downarrow\cdstrut\right.\@cdstrut\fi
	\ifx u#1
	\llap{$\scriptstyle#2$}\left\uparrow\cdstrut\right.\@cdstrut\fi
	\ifx r#1
	\mathop{\hbox to \cdsep{\rightarrowfill}}\limits^{#2}\fi
	\ifx l#1
	\mathop{\hbox to \cdsep{\leftarrowfill}}\limits^{#2}\fi
}
\catcode`\@=12

\title{Distinguished Representations of $\Gl{n}{F}{q}$}
\author{Guy Kapon}
\date{October 2021}

\begin{document}

\maketitle
\begin{abstract}
Let $\rmf{G} = \rmf{Gl}_{n}(K)$, and $\rmf{H} = \rmf{G}^{\sigma}$ for $\sigma$ an involution of the form $g\rightarrow aga^{-1}$, It is known that for $K =\bb{Q}_q$ any irreducible representation of $\rmf{G}$ with an $\rmf{H}$ invariant functional, is self dual, we prove an analogous result for $\bb{F}_{q}$.
\end{abstract}

\tableofcontents
\section{Introduction}
Relative representation theory is concerned with understanding representations of the form $C^{\infty}(G/H)$ or some other suitably chosen space of functions, In particular an interesting question is which irreducible representations sit inside $C^{\infty}(G/H)$.

This representation will be the induction from $H$ to $G$ of the trivial representation taken in the appropriate sense, and therefore we will have $\rmf{Hom}_{G}(\pi,C^{\infty}(G/H))=\rmf{Hom}_{H}(\pi|_{H},1_{H})$. In other words the number of times $\pi$ appears in $C^{\infty}(G/H)$ is the dimension of $H$ invariant functionals, in the case where this is non-zero $\pi$ is called $H$-distinguished.

An important family of $H$ are the ones that form symmetric pairs, If $G$ is an algebraic group, and $\theta$ an involution, we denote $H=G^{\theta}$ the fixed points of $\theta$, and we call the pair $(G,H)$ a symmetric pair (and $G/H$ a symmetric space). 


In \cite{prasad2015relative} Prasad made a conjecture, which was later generalized by Erez Lapid, for a necessary condition for a representation to be distinguished in the case of symmetric pairs.

\begin{conjecture}
(Lapid-Prasad) Let $G$ be a connected reductive algebraic group defined over a local non-archemidiean field $F$, $\theta:G\rightarrow G$ an involution defined over $F$, $\rmf{H}=G^{\theta}$. Let $\pi $ be a smooth irreducible representation of $G(F)$. If $\pi$ is $\rmf{H}-$distinguished then the L-packet\footnote{We will not define L-packets as we do not need them here, for an expository note see \cite{article}} of $\pi$ is invariant to the functor $\tau \rightarrow \tilde{\tau}^{\theta}$, where $\tilde{\tau}$ is the contragredient representation, and $\tau^{\theta}$ is the twist by $\theta$.
\end{conjecture}

This is also conjectured for the archimedean case, we only need to take the correct category of representations.

Several cases of this were proved: 
\begin{itemize}

\item For the pair $(\Gl{a+b}{Q}{p},\Gl{a}{Q}{p} \times \Gl{b}{Q}{p})$ this was proven in \cite{uniq}.

\item In \cite{Feigon2012OnRD} the case where $G =\mathrm{GL}_{n}(E)$ for $E/F$ a quadratic extension, and $\rmf{H}$ is a subgroup of unitary transformations was proved. 

\item The case of real Galois symmetric pairs was proved in \cite{2017}

\end{itemize}
Another natural setting where this conjecture can be stated is over finite fields, In this paper, we will prove this conjecture for the Jaquet-Rallis case over finite fields (of characteristic different than 2).

\begin{main theorem}
\label{main theorem}
Let $q$ be an odd prime power, denote $\rmf{G}=\Gl{n}{F}{q}$ and let $A\in \Gl{n}{F}{q}$ be a matrix such that $A^2=I$. Define $\rmf{H}=\{X\in \rmf{G}| AXA^{-1}=X\}$. Assume we are given an irreducible representation $\pi$ of $\rmf{G}$ s.t $\pi^{\rmf{H}}\neq 0$, then $\pi^*\cong \pi$
\end{main theorem}
\subsection{Ideas of The Proofs}
We start by using the result for $\bb{Q}_{q}$, from an irreducible cuspidal representation of $\Gl{n}{F}{q}$ we can get an irreducible cuspidal of $\Gl{n}{Q}{p}$, if we started with a distinguished one we will end up with a distinguished one, which is, therefore, self-dual, and since the process is injective the original representation is also self-dual.

The next step is to try and see what happens during parabolic induction, using Mackey theory and some standard manipulations we get that if a parabolic induction is distinguished then some Jacques restriction of it is also distinguished (with respect to a different involution).

We then use the theory of Zelevinsky to break the induction process to two steps, irreducibles that lie in the parabolic induction of a parabolic $\rho$ with itself $n$ times are called $\rho$-primary, and any representation is the parabolic induction of primary representations over different cuspidals.

We then prove the case of primary representations, and then using it we complete the proof for general representations.

\subsection{Structure of The Paper}
In \S \ref{sec:2} we review the notions we will use for $\rmf{Gl}_{n}$, and it's subgroups, this will mostly be used in \S\ref{sec:4}.

In \S \ref{sec:3} we prove the case of cuspidal representations using \cite{uniq}.

In \S \ref{sec:4} we prove the general case, starting by doing the "induction step", then reviewing the Zelevinsky PSH-algebra theory, and then using it to complete the proof.

\subsection{Acknowledgments}

I would like to thank Avraham Aizenbud, for suggesting this problem, helping us along the way, and teaching me mathematics. 

I was partially supported by ISF 249/17 and BSF 2018201 grants.
 \section{Preliminaries} \label{sec:2}
 \subsection{Notations}
Throughout the paper we fix an odd prime power $q$, In this section, we mainly explain the notions we use for subgroups and operations on $\Gl{n}{F}{q}$, when
 there is no confusion we will use $G_{n}$ for $\Gl{n}{F}{q}$.
 
 \begin{notion}
 We use $[n]$ for the set $\{0,1,...,n-1\}$, when working with $\rm{Gl}_{n}$ we consider those as the indexes of the matrix.
 \end{notion}
 \begin{definition}
 For an ordered partition of $[n]$ with sets $T=\{T_{0},T_{1},\cdots T_{m-1}\}$ we denote by $P(T),U(T),L(T)$ the corresponding standard parabolic, its unipotent subgroup, and it's levy quotient (which we consider inside $G_{n}$). We denote by $f_T:[n]\rightarrow [m]$ the function that sends $T_{i}$ to $i$.
 \end{definition}
 When there is no confusion we drop the $T$ from the notation.
 
 The groups can be written directly:
 $$ P(T) = \{A\in G_{n}| \forall_{i,j}: f_{T}(i)>f_{T}(j)\Rightarrow A_{i,j}=0\}$$
 $$ U(T) = \{I+A\in G_{n}| \forall_{i,j}: f_{T}(i)\geq f_{T}(j)\Rightarrow A_{i,j}=0\}$$
 $$ L(T) = \{A\in G_{n}| \forall_{i,j}: f_{T}(i)\neq f_{T}(j)\Rightarrow A_{i,j}=0\}$$

We abuse notation and denote by $T_{i}$ also the space generated by the coordinates of $T_{i}$. 
 \begin{notion}
 For $i \in [m]$ we denote $n_{i}=|T_{i}|$, we denote by $L_{i}=\Gl{n_{i}}{F}{q}$ where we think of it as the i'th factor in the levy group.
 \end{notion}
 
 Throughout the paper we will usually have some involution $\sigma$ fixed, our involutions are always conjugation maps, so we define:
 \begin{definition}
 for $A\in G_{n}$, $\sigma_{A}:G_{n}\rightarrow G_{n}$ is the conjugation by $A$, $\sigma_{A}(X)=A\cdot X\cdot A^{-1}$. $H_{A}$ is the subgroup fixed by $\sigma_{A}$, i.e the centraliser of $A$.
 \end{definition}
 
 We will have most of the time that $A$ normalizes the standard torus, in that case, $A$ is a permutation of the indices up to an element of the torus, and we denote that permutation $\tau_{A}$.

From a representation of $L$ we can construct a representation of $G_{n}$ by parabolic induction. Since $L$ is a product of groups we consider the parabolic induction as a multiplication of representations and denote it in a way that reflects this.
 \begin{definition}
 If $\pi_{i}\in rep(L_{i})=rep(G_{n_{i}})$ are representations, we denote $x = \boxtimes_{i\in [m]} \pi_{i}$ is the outer tensor product of them, which is a representation of $L$. We then consider it as representation of $P$ using the canonical map $P\rightarrow L$. With this notions we define: $$\prod_{i\in [m]} \pi_{i} = Ind_{P}^{G_{n}} x$$
 \end{definition}
 
 In the other direction, we have the Jacquet functor.
 \begin{definition}
 Let $\pi$ be a representation of $G_{n}$, we denote by $\jr{T}{\pi}$ the Jacquet restriction, which is defined by restricting to $P$ and then taking $U$ co-invariants, this is a representation of $L$.
     
 \end{definition}

 \section{The Cuspidal Case} \label{sec:3}
For our base case we use the result for $\bb{Q}_{q}$:
 \begin{theorem}[{\cite[Theroem~1.1]{uniq}}] \label{jaq-ral} Let $G_{n} = \rmf{GL}_{n}(K)$ where $K$ is a local non-Archimedean field of characteristic 0. let p,q be positive integers with $p+q=n$ and let $H_{p,n}$ be the subgroup of $G_{n}$ which is $G_{p}\times G_{q}$. if $\pi$ is an admissible irreducible representation of $G_{n}$, then:
 $$ \rmf{dim}(\InPr{\pi|_{H_{p,q}}}{1_{H_{p,q}}})\leq 1$$
 And if there is equality then $\pi$ is isomorphic to its contragredient representation.

 \end{theorem}
 
For us only the case of equality is relevant

The second thing we use is:
 \begin{lemma}
 There is an injective map between irreducible cuspidal representations of $\Gl{n}{F}{q}$ and irreducible cuspidal representations of $\Gl{n}{Q}{q}$.
 \end{lemma}
 This map takes a representation of $\Gl{n}{F}{q}$, considers it as a representation of $\Gl{n}{Z}{q}$, extends it trivially on the center of $\Gl{n}{Q}{q}$ and then inducts it to the whole group. This lemma is part of the theory around Moy-Prassad filtration and follows from \cite[Proposition 6.6]{jes-moy}, for completeness we give proof for the irreducibility and the injectivity (the cuspidality is unimportant for us).
 
 \begin{proof}
 Let $\pi$ be a cuspidal irreducible representation of $\Gl{n}{F}{q}$, let $\tilde{\pi}$ be it's extension to $K = \Gl{n}{Z}{q} \cdot Z(\Gl{n}{Q}{q})\cong \Gl{n}{Z}{q} \times \bb{Z}$ where $\bb{Z}$ acts trivially, and let $x = Ind_{K}^{G} \tilde{\pi}$ be the induction to $G = \Gl{n}{Q}{q}$.  
 Let $\pi',\tilde{\pi'},x'$ be the same for another irreducible cuspidal representation $\pi'$, let us calculate:
 $$ \InPr{x}{x'}$$
 Using mackeys theory we have:
 $$ \InPr{x}{x'} = \InPr{\tilde{\pi}}{x'|_{K}} = \sum_{g\in K\backslash G / K} \InPr{\tilde{\pi}}{Ind_{K\cap K^{g}}^{K}\tilde{\pi'}^{g}|_{K\cap K^{g}}}=$$
$$  = \sum_{g\in K\backslash G / K} \InPr{\tilde{\pi}|_{K\cap K^{g}}}{\tilde{\pi'}^{g}|_{K\cap K^{g}}}$$
Now by definition we know that matrices in $K_{1}=I + pM_{n}(\bb{Z}_{q})$ act trivially on $\tilde{\pi}$, and so if there is a map from $\tilde{\pi}|_{K\cap K^{g}}$ it must land on the $K_{1}\cap K^{g}$ invariants. translating to our case we see it must land on the $K\cap K_{1}^{g}$ invariants of $\tilde{\pi'}$, which is the same as the $(K\cap K_{1}^{g})/K_{1}$ invariants of $\pi'$.

Now it is well known that $(K\cap K^{g})/K_{1}$ is a parabolic subgroup, and $(K\cap K_{1}^{g})/K_{1}$ is it's unipotent subgroup, and since $\pi'$ is cuspidal, unless this is the trivial parabolic and unipotent subgroup, there are no invariant, in other words $\InPr{\tilde{\pi}|_{K\cap K^{g}}}{\tilde{\pi'}^{g}|_{K\cap K^{g}}}=0$ unless $K^g = K$, but since the normalizer of $\Gl{n}{Z}{q}$ is $K$ we see that this only happens when $g\in K$. 

In conclusion we see that:
$$\InPr{x}{x'} = \InPr{\tilde{\pi}}{\tilde{\pi'}} = \InPr{\pi}{\pi'}$$
In particular if $\pi,\pi'$ are different then so is $x,x'$, and since $\pi$ is irreducible then $\InPr{x}{x}=1$ and therefore $x$ is irreducible.
 \end{proof}
 \begin{remark}
Proving that $x$ is cuspidal can be done very similarly, as it amounts to calculating $\InPr{1_{U}}{x|_{U}}$ which again can be done by Mackey theory.
 \end{remark}
 
 Now we are ready to prove:
 \begin{theorem} \label{cus-case}
 Let $\pi$ be a cuspidal irreducible representation of $\Gl{n}{F}{q}$, $A\in \Gl{n}{F}{q}$ such that $\sigma_{A}$ is an involution of $\Gl{n}{F}{q}$, and assume that $\pi$ is $H_{A}$ distinguished, then $\pi$ is self dual.
 \end{theorem}
 
 \begin{proof}
 Let $\tilde{A}$ be a lift of $A$ to $\Gl{n}{Z}{q}$, and $H_{\tilde{A}}$ be the invarinats of $\sigma_{\tilde{A}}$ in $\Gl{n}{Q}{q}$. let us show that $x$ (from the lemma) is $H_{\tilde{A}}$ distinguished.
 
 $$ \InPr{x|_{H_{\tilde{A}}}}{1_{H_{\tilde{A}}}}\geq \InPr{ind_{H_{\tilde{A}}\cap K}^{H_{\tilde{A}}} (\tilde{\pi}|_{H_{\tilde{A}}\cap K})}{1_{H_{\tilde{A}}}} = \InPr{\tilde{\pi}|_{H_{\tilde{A}}\cap K}}{1_{H_{\tilde{A}}\cap K}}$$
 
Where the inequality is by Mackey's theory.

Now $K_{1}$ acts trivially on $\tilde{\pi}$, so we can consider this as representations mod $K_{1}$, and $(H_{\tilde{A}}\cap K)/K_{1}\subseteq H_{A}$, so the existence of such a map is followed from the fact that $\pi$ is $H_{A}$ distinguished.

Now since $x$ is distinguished we know by \ref{jaq-ral} that it is self-dual (admissibility of $x$ is obvious). However both the induction and trivial extension commute with duality, and since we proved that on irreducible cuspidal representations this operation is injective, this must mean $\pi \cong \pi^{*}$ as wanted.
 \end{proof} 
 \section{The General Case}
 \label{sec:4}
 Now that we know the result for cuspidal representations, our goal is to extend from cuspidal to general irreducible representations.
 
We start by showing in \ref{Res} that if a parabolic induction is distinguished, then there is some Jacques restriction which is also distinguished (with respect to a different involution).

In \ref{PrRep} we use the restriction to show that if a primary representation is distinguished, then the cuspidal representation below it is self-dual, we then use the theory in \cite{zel} to show that this implies that the primary representation is self-dual.

Finally, in \ref{EoP} we again use the restriction to now complete the general case.
\subsection{The Geometric Lemma}
Before we start we need some results about $\rmf{H}$, and how $P$ acts on $G_{n}/H$. these results come from \cite{GL}.

\begin{lemma} \label{geo}
Let $G$ be a connected reductive algebraic group over $k$,$\theta$ an involution of $G$ defined over $k$, $H=G^{\theta}$ the fixed points of $\theta$, and $P$ a minimal parabolic over $k$. Then $P$ contains a $\theta$ stable maximal $k$ split torus $A$, and every double coset $H(k)gP(k)$ contains an element $x$ such that $x\cdot \theta(x)^{-1} \in N(A(k))$. 
\end{lemma}

The $\theta$-stable torus part is lemma 2.4 of \cite{GL}, while the rest follows from chapter 6 (specifically 6.6).

\begin{remark}
Notice that if $P$ is any parabolic that contains a minimal $k$ parabolic then the lemma is still true.
\end{remark}

\subsection{A distinguished Restriction}
\label{Res}
  
In this section, our goal is to develop the tool that will let us prove the theorem for big representations using smaller ones which we already know the result for.

We fix some $n$ and an involution $\sigma=\sigma_{A}$ of $G_{n}$. We assume $A$ normalizes the standard torus (This is fine since $A$ is diagonizbele)

The exact statement we prove in this section is as follows:
\begin{lemma} \label{res-lem}
Let $T=\{T_{i}\}_{i\in [m]}$ be an ordered partition of $[n]$, let $\pi_{i}\in rep(L(T)_{i})$ be representations. Assume that $\pi = \prod_{i\in [m]} \pi_{i}$ is $H$ distinguished, then there exist some $A'$ that normalizes the torus,and partitions $T_{a,b}=\{i\in T_{a}|f(\tau_{A'}(i))=b\}$ of $T_{a}$, such that: $$\boxtimes_{a} \jr{T_{a,b}}{\pi_{a}}  $$ as a representation of $L(T_{a,b})$ is distinguished with respect to $\sigma_{A'}$.
\end{lemma}
\begin{remark}
Notice that this is a representation of the levy of $T_{a,b}$, which is a product of groups, and $\sigma_{A'}$ indeed act on it so the statement makes sense. While it does not act on every block separately, one can see that since $\sigma_{A'}$ is an involution it does act on pairs of blocks.
\end{remark}

We start by using Mackey theorem to get:

\begin{claim}
There is some $ PgH \in P\backslash G/H$ such that $x=\boxtimes_{a\in [m]} \pi_{a}$ is distinguished with respect to $H^{g}\cap P$.
\end{claim}
\begin{proof}
$\pi$ is $\rmf{H}$ distinguished which means:
$$ \InPr{\pi |_{\rmf{H}}}{{1}_{\rmf{H}}}  > \neq 0 $$
Now by definition $\pi = Ind_{P}^{G}(x)$, by mackey theory we get:
$$ \pi |_{\rmf{H}} = Ind_{P}^{G}( x )|_{\rmf{H}} = \bigoplus_{g\in P\backslash G/H} Ind_{H\cap P^g}^{H} (x^g|_{H\cap P^g}) $$
Therefore:
$$ \InPr{\pi |_{\rmf{H}}}{{1}_{\rmf{H}}} = \sum_{g\in P\backslash G/H} \InPr{Ind_{H\cap P^g}^{H} (x^g|_{H\cap P^g})}{{1}_{\rmf{H}}} = $$
$$  = \sum_{g\in P\backslash G/H} \InPr{x^g|_{H\cap P^g}}{{1}_{\rmf{H \cap P^g}} } = \sum_{g\in P\backslash G/H} \InPr{x|_{H^g\cap P}}{{1}_{\rmf{H^g \cap P}}} $$
And since the LHS is positive, then so is the RHS, so one of those terms is non-zero, which is exactly what we wished to prove.
\end{proof}

Now by \ref{geo} we know that any $g \in P\backslash G/H$ can be chosen such that $g^{-1} \sigma(g)=Q(g)$ is in $N(A)$.
 
 Now $a\in H^g$ if and only if:
 $$ gag^{-1} = \sigma(gag^{-1})$$
 which is the same as saying:
 $$ a = Q(g)\sigma(a)Q(g)^{-1}$$
 which means those are exactly matrices in $P$ which are invariant under this new involution, notice that $Q(g)$ normalizes the torus, so it is a permutation up to an element of the torus.
 
 so now we know that $x$ is distinguished by matrices that are preserved by this involution, call the invariant vector $v$,and call the corresponding premutiaion $\tau$, for $a,b \in [m]$ denote by $T_{a,b}$ all $i\in [n]$ such that  $f(i)=a,f(\tau(i))=b$.
 \begin{claim}
 Let $a,b,c\in [m]$ and assume $a<b$. Let $A$ be a matrix s.t $A_{i,j}=0$ unless $f(i)=f(j)=c$ and $f(\tau(i))=a,f(\tau(j))=b$. then $(I+A)v=v$.
 \end{claim}
 \begin{proof}
 consider $\sigma^{g}(A)$, it is only non-zero on $(i,j)$ s.t $f(i)=a<b=f(j)$ and therefore $1+\sigma^{g}(A)$ is in $U$, which acts trivially on all of $x$. Furthermore I claim that $\sigma^g (A)$ and $A$ commute, for this consider $T_{a,b}$ as a vector space spanned by those coordinates, $A$ is a map from $T_{c,a}$ to $T_{c,b}$ while $\sigma^g(A)$ is a map from $T_{a,c}$ to $T_{b,c}$, since we can't have that $T_{b,c}=T_{c,a}$ or $T_{c,b}=T_{a,c}$ then the composition of $A$ and $\sigma^g (A)$ is zero either way.

 Now denote $B = 1+A$, we see that $B\sigma^g(B)=\sigma^g(B) B = \sigma^g(B\sigma^g(B))$ and therefore $B\sigma^g(B)v=v$ but $\sigma^g(B)v=v$ so we get $Bv=v$
 \end{proof}

 
 \begin{corollary}
 If $A$ is a matrix s.t $A_{i,j}=0$ unless $f(i)=f(j)$ and $f(\tau(i))<f(\tau(j))$ then $(1+A)v=v$. 
 \end{corollary}
 \begin{proof}
 This group is obviously generated by the matrices of the previous claim.
 \end{proof}
 Now consider the levy group $L'$ corresponding to the $T_{a,b}$, obviously $L'\subseteq L$, and the unipotent factor is exactly the matrices of the corollary, so we conclude that the Jacques restriction from $L$ to $L'$ of $x$ is distinguished with respect to $\sigma^{g}$.

\subsection{PSH-algebras}
\label{PSH}

To complete the proof we need to review the way in which general represntaions are built up from cuspidal ones, for that we review the description in \cite{zel}.
\begin{definition}
A Hopf algebra over a comutative ring $K$ is a graded $K$ module $R=\bigoplus_{n\geq0} R_{n}$ with graded morphisms (over $K$) $m:R\otimes_{K} R\rightarrow R$, $m^*:R\rightarrow R\otimes_{K} R$, $e:K\rightarrow R$, $e^*:R\rightarrow K$ s.t $m,e$ turn $R$ to a ring, $m^*,e^*$ turns it into a co-algebra, and $m^*$ is a ring map.
\end{definition} 
In other words, a Hopf algebra is something that is both an algebra and a co-algebra, and both structures commute.
\begin{definition}
A PSH-algebra is a Hopf algebra over $\mathbb{Z}$ together with a free homogeneous basis for $R$ (as an abelian group), s.t $R_{0}\cong \mathbb{Z}$ with isomorphisms $e,e^*$, all the maps of the Hopf algebra take basis elements to a positive-sum of basis elements, and with respect to the inner product in which our basis is an orthonormal basis, $m,m^*$ and $e,e^*$ are adjoint pairs.
\end{definition}
One should think of examples of the form $R_{n}=R(G_{n})$ where $R(G)$ is the ring of virtual representations, the free basis is just irreducible representations,  the multiplication and co-multiplication will be a form of induction and restriction, the positivity will correspond to the fact that those operations take actual representations to actual representations, the adjointness will correspond to the  adjointness  (as functors) of induction and restriction, and the most non-trivial part will be the fact that multiplication and co-multiplication commute, which will follow from Mackey's theory. 
\begin{claim}
If we denote $G_{n}=S_{n}$ or $G_{n}=\Gl{n}{F}{q}$ for all $n$, then $R=\bigoplus_{n\geq 0} R(G_{n})$ is a PSH-algebra. In the case of $S_{n}$ the multipication and co-multipication comes from the natural inclusion $S_{n}\times S_{m} \subseteq S_{n+m}$ and the usual induction and restriction. In the case of $\Gl{n}{F}{q}$ it is the same but with parabolic induction and Jacques restriction.
\end{claim}

It turns out that PSH-algebras have a very simple structure theorem.
\begin{definition}
    A basis element of a PSH-algebra is called irreducible, an element $x$ s.t $m^*(x)=x\otimes 1 + 1 \otimes x$ is called primitive. 
\end{definition}
In the example of $\Gl{n}{F}{q}$ we see that the basis elements are exactly irreducible representations, and it's easy to see that the definition of primitive is the same as the usual definition of a cuspidal representation.
\begin{theorem}
Let R be a PSH-algebra, denote $\mathscr{C}$ as the set of irreducible primitive elements, then $R\cong \bigotimes_{\rho \in \mathscr{C}} R(\rho)$ where $R(\rho)$ is a PSH-algebra with only one irreducible primitive element (and is in fact the PSH-algebra generated by $\rho$).  
\end{theorem}
In the theorem one has to be careful, as there is an infinite tensor product, we mean by it the colimit over finite tensors, where maps are induced by the product with the unit.

The algebra generated by $\rho$ is easily described, for each $\rho^n$ we write it as a sum of basis elements, and then take the space generated by those basis elements.

This theorem in particular says that:
\begin{claim}
For a PSH-algebra, we call an element $\rho$-primary if it is in $R(\rho)$, then every irreducible element can be uniquely written as a product of primary irreducible elements over different $\rho$.
\end{claim}

If we apply this to the PSH-algebra of $\Gl{n}{F}{q}$ we get the same statement for representations of $\Gl{n}{F}{q}$, where the product is parabolic induction, and the $\rho$ are irreducible cuspidal representations.

The theorem is complemented by:
\begin{theorem}
All PSH-algebra with one irreducible primitive element of degree $1$ are isomorphic, the isomorphism is not unique, but there are only two isomorphisms. 
\end{theorem}
First, notice that the restriction on the degree is not essential, we saw that $R(\rho)$ only has elements of degree $n\cdot deg(\rho)$, so we can just divide all the degrees by $deg(\rho)$. Also notice that for $S_{n}$ there are no irreducible primitive elements except the trivial representation of $S_{1}$, so this gives an example of this unique algebra.

We can also describe the only automorphism:
\begin{claim}
If $R$ has only one primitive irreducible element $\rho$, then $\rho^2 = x_{2}+y_{2}$ where $x_{2},y_{2}$ are irreducible, the unique non trivial automorphism of $R$ swaps $x_{2},y_{2}$, In particular an isomorphism of this algebra is defined by what it does to $y_{2}$.
\end{claim}

For example, one can see that for $S_{n}$ tensoring with the sign representation is an automorphism, and it indeed swaps $x_{2},y_{2}$ which in this case are exactly the trivial and sign representations of $S_{2}$.

Another example of an automorphism comes from duality, for $S_{n}$ this is of course the identity automorphism. For $\Gl{n}{F}{q}$ this is not the case, we get a map $^*:R(\rho)\rightarrow R(\rho^*)$, and not every cuspidal is self-dual, however, we do claim:
\begin{claim}
\label{C:id}
If $\rho$ is self dual, then $^*:R(\rho)\rightarrow R(\rho^*)=R(\rho)$ is the identity.
\end{claim}
\begin{proof}
By what we said before it's enough to see what we do on the irreducible representations inside $\rho^2$, The idea is that we can distinguish between the two irreducible inside using a Whitaker character, specifically let $\chi:\mathbb{F}_{q} \rightarrow \mathbb{C}^*$ be the non-trivial real character and extend it to $U$ by $\chi(A)=\sum_{1\leq i < n} \chi(a_{i,i+1})$, for any representation we can look at the dimension of the $\chi$ equivariant vectors in it, it is known that for $\rho^{n}$ there is only one, in particular, this means exactly one of the two irreducibles in $\rho^2$ has such a vector, we call that one $y_{2}$. now notice that $y_{2}^*$ has a $\chi^*$ equivariant functional,  but $\chi^*=\chi$ since it's real, and since we are working with finite groups having an equivariant functional or vector is the same (both correspond to a map from or to the representation $\chi$). So $y_{2}^*$ has a $\chi$ equivariant vector and is, therefore, $y_{2}$, so we are done.
\end{proof}
\begin{remark}
The fact that duality is an isomorphism of the PSH-algebra is not immediate, one needs to show that it commutes with the multiplication and co-multiplication, for the usual induction and restriction this is obvious, for parabolic induction this is also easy, for Jacques restriction however this is more delicate. indeed for $\pi \in rep(\Gl{n}{F}{})$ we need to take the dual of $(\pi |_{P})^U$, which is $((\pi^*)|_{P})_{U}$ as for finite-dimensional spaces duality takes invariants to co-invariants and vice-versa, however since we work over $\mathbb{C}$ and with finite groups, we have the averaging operator which identifies invariants with co-invariants.
\end{remark}
Finally, the last claim can be restated as:
\begin{corollary}
\label{pri-du}
A $\rho$-primary representation is self-dual iff $\rho$ is.
\end{corollary}
\begin{proof}
If $\pi \in R(\rho)$ is self dual, then $\pi=\pi^*\in R(\rho^*)$ so we must have $\rho=\rho^*$. On the other hand if $\rho=\rho^*$ then \ref{C:id} says that $^*$ is identity on $R(\rho)$ so we are done.
\end{proof}
\subsection{Primary Representations}
The goal of the section is to prove the primary case:
\begin{theorem}
If $\pi$ is a $\rho$-primary representation, and $\pi$ is $\rmf{H}$ distinguished, then $\pi$ is self-dual.
\end{theorem}

Because of \ref{pri-du} we only have to prove that $\rho$ is self-dual, and since any primary representation is a sub-representation of $\rho^n$ for some $n$ it's enough to prove:

\begin{theorem}
If $\rho^n$  is $H$ distinguished then $\rho$ is self dual.
\end{theorem}
\begin{proof}
We apply \ref{res-lem} on the product $\rho^n$, we get that there is a torus normalising matrix $Q$, it's permutation of the indexes is $\tau$, and we consider the levi that corresponds to the spaces $T_{a,b}=\{i|f(i)=a,f(\tau(b)\}$. The Jacques restriction to it is distinguished with respect to the $Q$-centraliser, however $\rho$ is cuspidal, and all it's non-trivial restrictions are zero, we see that $T_{a,b}$ is either empty or all of $T_{a}$. since $\tau$ is an involution this means that $[m]$ is split to pairs and singletons, where in the singletons $T_{a,a}\neq \phi$ and in the pairs $T_{a,b}\neq \phi \neq T_{b,a}$. so now we have two cases.

Case 1: we have $T_{a,a}\neq \phi$. In this case the $\rho$ siting in $T_{a}$ is distinguished by the restriction of $\sigma^Q$ to $T_{a}$, but since it is cuspidal, we can use \ref{cus-case} and therefore $\rho$ is self dual.

Case 2: We have $T_{a,b}\neq \phi \neq T_{b,a}$, in this case we have $\sigma^Q$ takes $L_{a}$ to $L_{b}$ and vice versa. If we write this map as $\sigma':\Gl{r}{F}{q}\rightarrow \Gl{r}{F}{q}$, then the invariant matrix are exactly $(A,\sigma'(A))$. restricting $\rho \boxtimes \rho$ we get $\rho \otimes \sigma'(\rho)$ where here this is the tensor in $rep(\Gl{n}{F}{q})$. but since $\sigma'$ is essentially the restriction of a conjugation we can see that it is a conjugation map, and conjugation doesn't change representations, so we get that $\rho\otimes\rho$ is $\Gl{n}{F}{q}$ distinguished. Therefore we can write $0\neq (\rho \otimes \rho)^{\Gl{n}{F}{q}}=(\rmf{Hom}(\rho,\rho^*))^{\Gl{n}{F}{q}}$, which means we have a non-zero equivariant map from $\rho$ to $\rho^*$, but since they are irreducible this means they are isomorphic, as wanted.
\end{proof}
\label{PrRep}
\subsection{End of Proof}
\label{EoP}

Now we are only left with proving:
\begin{theorem}
If $\pi$ is a product of $\rho_{i}$-primary representation $\pi_{\rho_{i}}$, and $\pi$ is $\rmf{H}$ distinguished, then $\pi$ is self dual.
\end{theorem}
The proof is done in a very similar way to the previous case.
\begin{proof}
Again we apply \ref{res-lem} on the product, so we have a matrix $Q$, a permutation $\tau$ and the partition $T_{a,b}$ as before.

For each $a$, if $T_{a,a}$ is non-empty, then we have some Jacques restriction of $\pi_{\rho_{a}}$ which is distinguished, since a restriction of a $\rho_{a}$ primary representation is still primary, this means that $\rho_{a}$ is self dual.
On the other hand if $T_{a,b}$ is not empty (and in particular $T_{b,a}$ is also non-empty) then we get that the product of a Jacques restriction of $\pi_{\rho_{a}}$ and of $\pi_{\rho_{b}}$ is distinguished. just like in the previous case this means that these restrictions are dual to each other, and therefore $\rho_{a}$ is dual to $\rho_{b}$.

Combining these two results we see that we cannot have $T_{a,b},T_{a,c}$ both non-empty for $b\neq c$, as then $\rho_{b}=\rho_{a}^{*}=\rho_{c}$, So we see that again the indexes are divided to singletons and pairs.

The $\rho_{a}$ such that $T_{a,a}$ is non empty are self dual and therefore also $\pi_{\rho_{a}}$ is self dual.
If $T_{a,b}$ is non empty, then since the restrictions don't do anything, we get $\pi_{\rho_{a}}^{*} = \pi_{\rho_{b}}$.

And since duality commutes with the product, we get that $\pi$ is self-dual.
\end{proof}
\medskip
\printbibliography


\end{document}